\documentclass{amsart}

\usepackage{epsfig}
\usepackage{amscd}
\usepackage{amsmath, amssymb, amsthm}
\usepackage{graphics}
\usepackage{color}
\newtheorem*{main-theorem}{Main Theorem}
\newtheorem{proposition}{Proposition}[section]
\newtheorem{theorem}{Theorem}
\newtheorem{lemma}[proposition]{Lemma}

\theoremstyle{definition}

\newtheorem*{remark}{Remark}

\numberwithin{equation}{section}

\def\OP{{\mathrm{OP}}}
\def\ZZ{{\mathbb Z}}
\def\reals{{\mathbb R}}
\def\cx{{\mathbb C}}

\def\Ci{{\mathcal C}^\infty}

\def\WF{\mathrm{WF}_h\,}

\def\supp{\mathrm{supp}\,}

\def\id{\,\mathrm{id}\,}

\def\O{{\mathcal O}}

\def\s{{\mathcal S}}

\def\neigh{\mathrm{neigh}\,}
\def\Op{\mathrm{Op}\,}

\def\esssupp{\text{ess-supp}\,}

\def\phi{\varphi}

\def\be{\begin{eqnarray*}}
\def\ee{\end{eqnarray*}}
\def\ben{\begin{eqnarray}}
\def\een{\end{eqnarray}}

\def\L2R{L_{\text{Rest}}^2}

\def\tU{\widetilde{U}}
\def\tF{\widetilde{F}}
\def\tchi{\tilde{\chi}}

\begin{document}
\title[Isomorphisms]{Isomorphisms between Algebras of Semiclassical
  Pseudodifferential Operators}
\author{Hans Christianson}
\address{Department of Mathematics, Massachusetts Institute of
Technology, 77 Mass. Ave., Cambridge, MA 02139, USA}
\email{hans@math.mit.edu}

\begin{abstract}
Following the work of Duistermaat-Singer \cite{DS} on isomorphisms of
algebras of global pseudodifferential operators, we classify
isomorphisms of algebras of microlocally defined semiclassical
pseudodifferential operators.  Specifically, we show that any such
isomorphism is given by conjugation by $A = BF$, where $B$ is a
microlocally elliptic semiclassical pseudodifferential operator, and
$F$ is a microlocal $h$-FIO associated to the graph of a local symplectic transformation.
\end{abstract}
\maketitle


\section{Introduction}
In the study of pseudodifferential operators on manifolds, there are
two important regimes to keep in mind.  The first is a global study of
pseudodifferential operators defined using the local Fourier transform
on the cotangent bundle.  If $X$ is a compact smooth manifold and $T^*X$ is
the cotangent bundle with the local coordinates $\rho = (x, \xi)$, we study pseudodifferential operators with
principal symbol homomgeneous at infinity in the $\xi$ variables.  Let
$Y$ be another compact smooth manifold of the same dimension as $X$,
and suppose there is an algebra isomorphism from the algebra of all
pseudodifferential operators on $X$ (filtered by order) to the same algebra on $Y$, and
suppose that isomorphism preserves the order of the operator.  Then
Duistermaat-Singer \cite{DS} have shown that this isomorphism is
necessarily given by conjugation by an elliptic Fourier Integral
Operator (FIO).

The other setting is the semiclassical or ``small-$h$'' regime.  One can study globally
defined semiclassical pseudodifferential operators, but many times it
is meaningful to study operators which are microlocally defined in
some small set (see
\S \ref{preliminaries} for definitions).  Then we think of the $h$
parameter as being comparable to $|\xi|^{-1}$ in the global,
non-semiclassical regime.  Thus the study of small $h$ asymptotics in the
microlocally defined regime should correspond to the study of high
frequency asymptotics in the global regime.  We therefore expect a
similar result to that presented in \cite{DS}, although the techniques
used in the proof will vary slightly.

Let $X$ be a smooth manifold, $\dim X = n \geq 2$, and assume $U
\subset T^*X$ is an open set.  Let $Y$ be
another smooth manifold, $\dim Y = n$, and let $V \subset T^*Y$.  Let $\Psi^{0} / \Psi^{-\infty}(U)$ denote the algebra of 
semiclassical pseudodifferential operators defined microlocally in
$U$ filtered by the order in $h$, and similarly for $V$ (see \S \ref{preliminaries} for definitions).
\begin{theorem}
\label{main-theorem}
Suppose 
\be
g: \Psi^0 / \Psi^{-\infty}(U) \to \Psi^0 / \Psi^{-\infty}(V)
\ee
is an order preserving algebra isomorphism.  For every $\widetilde{U}
\Subset U$ open and precompact, there is a
symplectomorphism
\be
\kappa: \overline{\widetilde{U}} \to \overline{\kappa (
\widetilde{U} )}
\ee
and $h_0>0$ such that, if $F$ is the $h$-FIO associated to
$\kappa$, for all $0 < h < h_0$ and all $P \in \Psi^0 /
\Psi^{-\infty}(U)$ we have
\ben
\label{main-theorem-st}
g(P) = BF P F^{-1} B^{-1} \text{ microlocally in } \kappa(\widetilde{U}) \times \kappa(\widetilde{U}),
\een
where $B \in \Psi^0(V)$ is elliptic on $\kappa( \tU )$. 
\end{theorem}

To put Theorem \ref{main-theorem} in context, we observe that
every algebra homomorphism of the form \eqref{main-theorem-st} is an
order preserving algebra isomorphism, according to Proposition
\ref{global-egorov} in \S \ref{preliminaries}.

Automorphisms of algebras of pseudodifferential operators have also
been studied in the context of the more abstract Berezin-Toeplitz
quantization in \cite{Zel}.

{\bf Acknowledgements.}  The author would like to thank Maciej Zworski
for suggesting this problem and many helpful conversations.  This work
  was started while the author was a graduate student in the Mathematics
  Department at UC-Berkeley and he is very grateful for the support
  received 
  while there.


\section{Preliminaries}
\label{preliminaries}
 Let $\Ci(T^*X)$
denote the algebra of smooth, $\cx$-valued functions on $T^*X$, and
define the global symbol classes
\be
\s^m(T^*X)  = \left\{ a \in \Ci \left( (0,h_0]_h; \Ci(T^*X) \right) :
  |\partial^\alpha a| \leq C_\alpha h^{-m} \right\}.
\ee

We define the
essential support of a symbol by complement:
\be
\lefteqn{\esssupp_h (a) =} \\
&& =  \complement \left\{ (x, \xi) \in T^*X : \left| \partial^\alpha a \right|
\leq C_\alpha h^N \,\, \forall N \text{ and } \forall (x', \xi') \text{
  near } (x, \xi) \right\}.
\ee


By multiplying elements of $S^m(T^*X)$ by an appropriate cutoff in $\Ci_c(U)$, we may think of symbols as being
microlocally defined in $U$, and define the class of symbols with
essential support in $U$
\be
\s^m(U) = \left\{ a \in \Ci \left( (0,1]_h; \Ci_c(U)  \right) :
  |\partial^\alpha a| \leq C_\alpha h^{-m} \right\}.
\ee
We write $\s^m = \s^m(U)$ when there is no ambiguity.  We can think of elements of $\s^m$ as formal power series in $h$:
\be
a(x, \xi; h) = \sum_{j=-m}^\infty h^j a_j(x, \xi;h),
\ee
where each $a_j$ is in $\Ci_c(U)$ and has derivatives of all orders
bounded in $h$.

We have the corresponding spaces of pseudodifferential operators
$\Psi^m(U)$ acting by the local formula (Weyl calculus)
\be
\Op_h^w(a)u(x) = \frac{1}{(2 \pi h)^n} \int \int a \left( \frac{x + y}{2}, \xi; h \right) 
e^{i \langle x-y, \xi \rangle / h }u(y) dy d\xi.
\ee
For $A = \Op_h^w(a)$ and $B = \Op_h^w(b)$, $a \in \s^{m}$, $b \in \s^{m'}$ we have the composition 
formula (see, for example, the review in \cite{DiSj})
\begin{eqnarray}
\label{Weyl-comp}
A \circ B = \Op_h^w \left( a \# b \right),
\end{eqnarray}
where
\begin{eqnarray}
\label{a-pound-b}
\s^{ m+m'} \ni a \# b (x, \xi) := \left. e^{\frac{ih}{2} \omega(Dx, D_\xi; D_y, D_\eta)} 
\left( a(x, \xi) b(y, \eta) \right) \right|_{{x = y} \atop {\xi = \eta}} ,
\end{eqnarray}
with $\omega$ the standard symplectic form.  Observe $\#$ preserves
essential support in the sense that if $\esssupp_h(a) \cap
\esssupp_h(b) = \emptyset$, then $a \# b = \O(h^\infty)$.  We define
the wavefront set of a pseudodifferential operator $A = \Op_h^w(a)$ as
\be
\WF (A) = \esssupp_h(a),
\ee
so that $\Psi^m(U)$ is the class of pseudodifferential operators with
wavefront set contained in $U$.  We denote
\be
\Psi^0(U) & :=&  \bigcup_{m \leq 0} \Psi^m(U) \text{ and} \\
\Psi^{-\infty}(U) & :=& \bigcap_{m \in \ZZ} \Psi^m(U).
\ee

We will need the definition of microlocal equivalence of operators.  Suppose 
$T: \Ci(X) \to \Ci(X)$ and that for any seminorm $\| \cdot \|_1$ on 
$\Ci(X)$ there is a second seminorm $\| \cdot \|_2$ on $\Ci(X)$ such that 
\begin{eqnarray*}
\| Tu\|_1 = \O(h^{-M_0})\|u \|_2
\end{eqnarray*}
for some $M_0$ fixed.  Then we say $T$ is {\it semiclassically tempered}.  We assume for the rest of 
this paper that all operators satisfy this condition (see
\cite[Chap. 10]{EvZw} for more on this).  Let $U,V \subset T^*X$ be open pre-compact sets.  
We think of operators defined microlocally near $V \times U$ as equivalence classes of tempered operators.  
The equivalence relation is
\begin{eqnarray*}
T \sim T' \Longleftrightarrow A(T-T')B = \O(h^\infty): \mathcal{D}'\left( X \right) \to \Ci 
\left(X\right)
\end{eqnarray*}
for any $A,B \in \Psi_h^{0,0}(X)$ such that 
\begin{eqnarray*}
&& \WF (A) \subset \widetilde{V}, \quad \WF (B) \subset \widetilde{U}, \,\, \text{with} \,\, \widetilde{V}, 
\widetilde{U} \,\, \text{open and } \\
&& \quad \quad \overline{V} \Subset \widetilde{V} \Subset T^*X, \quad \overline{U} \Subset 
\widetilde{U} \Subset T^*X.
\end{eqnarray*}
In the course of this paper, when we say $P=Q$ {\it microlocally} near $U \times V$, we mean for any $A$, $B$ as above,
\begin{eqnarray*}
APB - AQB = \O_{L^2 \to L^2}\left( h^\infty \right),
\end{eqnarray*}
or in any other norm by the assumed pre-compactness of $U$ and $V$.  Similarly, we say $B = T^{-1}$ on $V \times V$ if 
$BT = I$ microlocally near $U \times U$ and $TB = I$ microlocally near
$V \times U$.  Thus 
\be
\Psi^0 / \Psi^{-\infty} (U) 
\ee
is the algebra of bounded semiclassical pseudodifferential operators defined
microlocally in $U$ modulo this equivalence relation.  It is
interesting to observe that this equivalence relation has a different
meaning in the high-frequency regime.  There, $\Psi^{-\infty}(X)$
corresponds to smoothing operators, although they may not be ``small''
in the sense of $h \to 0$.

We have the principal symbol map
\begin{eqnarray*}
\sigma_h : \Psi^m ( U) \to \s^{m} / \s^{m-1} 
(U),
\end{eqnarray*}
which gives the left inverse of $\Op_h^w$ in the sense that 
\begin{eqnarray*}
\sigma_h \circ \Op_h^w: \s^{m}(U) \to \s^{m}/\s^{ m-1} (U)
\end{eqnarray*}
is the natural projection.

We will use the following
well-known semiclassical version of Egorov's theorem (see \cite{Ch, Ch2} or
\cite{EvZw} for a proof).

\begin{proposition}
\label{AF=FB}
Suppose $U$ is an open neighbourhood of $(0,0)$ and $\kappa: \overline{U} \to \overline{U}$ is a symplectomorphism fixing $(0,0)$.  
Then there is a unitary operator $F : L^2 \to L^2$ such that for all $A = \Op_h^w(a)$,
\begin{eqnarray*}
AF = FB \,\, \text{microlocally on}\,\, U \times U,
\end{eqnarray*}
where $B = \Op_h^w(b)$ for a Weyl symbol $b$ satisfying
\begin{eqnarray*}
b = \kappa^* a + \O(h^2).
\end{eqnarray*}
$F$ is microlocally invertible in $U \times U$ and $F^{-1} A F = B$ microlocally in $U \times U$.
\end{proposition}

Observe that Proposition \ref{AF=FB} implicitly identifies a
neighbourhood $U$ with its coordinate representation.  To make a
global statement, we use the following Lemma.

\begin{lemma}
\label{cross-term-lemma}
Let $U_1, U_2 \subset \reals^{2n}$ be open sets with $H^1(U_j, \cx) =
\{0 \}$, $j = 1,2$.  Assume $V:= U_1 \cap U_2 \neq \emptyset$ and let
$\tU$ be a neighbourhood of $U_1 \cup U_2$.  Suppose 
\be
\kappa : \tU \to \kappa(\tU) \subset \reals^{2n}
\ee
is a symplectomorphism and let $F_j$ be the quantization of
$\kappa|_{U_j}$, $j = 1,2$, as in Proposition \ref{AF=FB}.  Then 
\be
&& F_1 F_2^{*} = \id + \O(h^2) \text{ microlocally near } \kappa(V)
\times \kappa(V) \text{ and } \\
&& F_1^{*} F_2 = \id + \O(h^2) \text{ microlocally near } V \times V
\ee
as pseudodifferential operators.  
\end{lemma}

\begin{proof}
From \cite[Corollary 3.4]{Ch2} we can for $0 \leq t \leq 1$ find a family of
symplectomorphisms $\kappa_t: \reals^{2n} \to \reals^{2n}$, a
Hamiltonian $q_t$, and linear operators $\tF_j(t): L^2( \reals^n) \to
L^2( \reals^n)$, $j = 1,2$ satisfying:
\be
\kappa_0 = \id, && \kappa_1 |_{\neigh (V)} = \kappa, \\
\frac{d}{dt} \kappa_t & = & (\kappa_t)_* H_{q_t},
\ee
and if $Q_t = \Op_h^w (q_t)$ is the quantization of $q_t$, the $\tF_j$ satisfy
\be
h D_t \tF_j(t) + \tF_j(t) Q(t) = 0, \,\, (0 \leq t \leq 1) \\
\tF_j(1) = F_j,
\ee
and $\tF_j(0) = \id + \O(h^2)$ as a pseudodifferential operator.  The
adjoints satisfy
\be
h D_t \tF_j^*(t) -  Q(t) \tF_j^*(t) = 0, \,\, (0 \leq t \leq 1) \\
\tF_j^*(1) = F_j^*,
\ee
and $\tF_j^*(0) = \id + \O(h^2)$ as a pseudodifferential operator.  A
calculation shows $\tF_1 \tF^*_2$ and $\tF_1^* \tF_2$ are constant.
The conclusion of the Lemma holds at $t=0$, so it holds at $t=1$
as well.
\end{proof}

The next proposition is a more global version of Proposition
\ref{AF=FB}.  
\begin{proposition}
\label{global-egorov}
Suppose $U \subset T^*X$ is an open set and $\kappa: \overline{U} \to
\overline{\kappa U}$ is a symplectomorphism.  
Then for any $\tU \Subset U$ open and precompact, there is a linear
operator $F : L^2(X) \to L^2(X)$, microlocally invertible in $\tU
\times \kappa( \tU )$ such that for all $A = \Op_h^w(a)$,
\be
F^{*} A F = B \text{ microlocally in } \tU \times \tU,
\ee
for $B = \Op_h^w(b)$ for a Weyl symbol $b$ satisfying
\begin{eqnarray*}
b = \kappa^* a + \O(h^2).
\end{eqnarray*}
\end{proposition}

\begin{proof}
The idea of the proof is to glue together operators from Proposition
\ref{AF=FB} with a partition of unity.  Let 
\be
1 = \sum_j \chi_j
\ee
be a partition of unity of $U$ so that $H^1(\supp \chi_j, \cx) = \{0\}$ 
for each $j$.  Let $U_j = \supp \chi_j \cap U$, $V_j = \kappa (U_j)$,
and let $F_j$ be the quantization of $\kappa|_{U_j}$ as in Proposition
\ref{AF=FB}.  Set $F = \sum_j  F_j \chi_j^w$, where $\chi_j^w =
\Op_h^w(\chi_j)$, so that $F^* = \sum_k \chi_k^w F_k^* $.  
We first verify:
\be
F^* F & = & \sum_{j,k} \chi_k^w F_k^* F_j \chi_j^w \\
& = & \sum_{j,k} \chi_k^w (1 + \O_{j,k}(h^2)) \chi_j^w \\
& = & 1 + \O(h^2)
\ee
microlocally on $\tU$ since $\tU$ is covered by finitely many
of the $U_j$s.  Further,
\be
F F^* & = & \sum_{j,k}  F_j \chi_j^w \chi_k^w F_k^* \\
& = & \sum_{j,k} \Op((\kappa^{-1})^* \chi_j + \O_j(h^2)) F_j F_k^*
\Op((\kappa^{-1})^* \chi_k + \O_k(h^2)) \\
& = & 1 + \O(h^2)
\ee
as above.  Hence $F^*$ is an approximate left and right inverse
microlocally on $\overline{\tU} \times \overline{\kappa ( \tU )}$ so
$F$ is microlocally invertible.

Now for each
$j$, choose $\tchi_j \in \Ci_c(T^*X)$ satisfying $\tchi_j \equiv 1$ on
$\supp \chi_j$ with support in a slightly larger set so that 
\be
\sum_j \tchi_j \leq C \text{ on } \tU, 
\ee
for $C>0$ fixed.  We
calculate for $A = \Op_h^w(a)$:
\be
F^* A F & = & \sum_{j,k} \chi_k^w F_k^* A F_j \chi_j^w \\
& = & \sum_{j,k} \chi_k^w F_k^* F_j B_j \chi_j^w,
\ee
where $B_j = \Op_h^w(b_j)$ for a symbol 
\be
b_j = (\kappa^* a + \O(h^2)) \tchi_j.
\ee
Then from Lemma \ref{cross-term-lemma} we have
\ben
\label{FAF}
F^* A F  =  \sum_{j,k} \chi_k^w B_j (1 + \O_{j,k}(h^2)) \chi_j^w ,
\een
and since we can cover $\overline{\tU}$ with finitely many of the
$U_j$, the error in \eqref{FAF} is $\O(h^2)$ microlocally on $\tU$ and
the Proposition follows.
\end{proof}

\begin{remark}
This notion of quantization of symplectic transformations is a
constructive version of the more general definition due to
H\"ormander-Melrose \cite{Hor1, Hor, Mel} as an integral
operator with a distribution kernel 
supported on the Lagrangian submanifold associated to the symplectic
relation (see also \cite{duistermaat} and the recent semiclassical treatment
in \cite{Ale}).
\end{remark}

Let $Y$ be another smooth manifold of the same dimension as $X$, and
let $V \subset T^*Y$ be a non-empty, pre-compact, open set.  We say 
\be
g: \Psi^0/\Psi^{-\infty}(U) \to \Psi^0/\Psi^{-\infty}(V)
\ee
is an order preserving algebra isomorphism (of algebras filtered by
powers of $h$) if
\be
g(\Psi^m(U)) = \Psi^m(V), \,\, g^{-1}(\Psi^m(V)) = \Psi^m(U),
\ee
and for every $A, A' \in \Psi^m(U)$, $B \in \Psi^{m'}(U)$,
\be
g(A + A') & = & g(A) + g(A') \bmod \Psi^{-\infty}(V), \\
g(AB) & = & g(A) g(B) \bmod \Psi^{-\infty}(V).
\ee


\section{The Proof of Theorem \ref{main-theorem}}
We break the proof of Theorem \ref{main-theorem} into several
lemmas.  
\begin{lemma}
The maximal ideals of $\s^0/\s^{-1}(U)+ \cx$ are either of the form 
\ben
\label{m-rho}
\mathcal{M}_\rho := \left\{ p \in \s^0/\s^{-1}(U)  : p(\rho)
= 0, \,\, \rho \in U \right\},
\een
or 
\ben
\label{m-bdy}
\mathcal{M}_{\partial U}:= \Ci_c(U) + \O(h^\infty) \Ci (U).
\een
\end{lemma}
\begin{proof}
Clearly for each $\rho \in U$, $\mathcal{M}_\rho$ is a maximal ideal.
Also, $\mathcal{M}_{\partial U}$ is maximal, since any ideal
$\mathcal{M}$ satisfying 
\be
\mathcal{M}_{\partial U} \subsetneq \mathcal{M}
\ee
must contain a constant, and therefore is equal to $\s^0/\s^{-1}(U)+\cx$.  Suppose
$\mathcal{M}$ is another maximal ideal which is not of the form
\eqref{m-rho} for any $\rho \in U$.  Then for each point $\rho
\in U$, there is $a_\rho \in \mathcal{M}$ such that
$a_\rho(\rho) \neq 0$.  Further, by multiplying by a (positive or
negative) constant if necessary, we may assume for each $\rho$ there is a neighbourhood
$U_\rho$ of $\rho$ such that $\left. a_\rho \right|_{\overline{U_\rho}}
\geq 1$.  Let $a(x, \xi) \in \mathcal{M}_{\partial U}$, and let 
\be
K =\esssupp_h (a) \Subset U.
\ee
As $K$ is compact, we can cover it with finitely many of the $U_\rho$, 
\be
K \subset U_{\rho_1} \cup \cdots \cup U_{\rho_m},
\ee
and 
\be
b:= \sum_{j=1}^m a_{\rho_j} \in \mathcal{M}
\ee
satisfies $b \geq 1$ on $K$.  Thus $a/b \in \mathcal{M}_{\partial U}$
implies 
\be
a = \left( \frac{a}{b} \right) b \in \mathcal{M}.
\ee
Thus $\mathcal{M}_{\partial U} \subset \mathcal{M}$.  But
$\mathcal{M}_{\partial U}$ is maximal, so either $\mathcal{M} =
\mathcal{M}_{\partial U}$ or $\mathcal{M} = \s^0/\s^{-1}(U)+\cx$.
\end{proof}

The following three lemmas are a semiclassical version of \cite{DS}
with a few modifications to the proofs.  
\begin{lemma}
\label{kappa-construction}
Suppose $g: \Psi^{0}/\Psi^{-\infty}(U) \to \Psi^0 /
\Psi^{-\infty}(V)$ is an order preserving algebra isomorphism.  Then
there exists a diffeomorphism $\kappa : U \to V$.
\end{lemma}
\begin{proof}
We first ``unitalize'' our algebra of pseudodifferential operators by
adding constant multiples of identity.  That is, let 
\be
\tilde{\s}^m(U) = \left\{ a \in \Ci \left( (0,1]_h; \Ci_c(U) + \cx  \right) :
  |\partial^\alpha a| \leq C_\alpha h^{-m} \right\},
\ee
and let $\widetilde{\Psi}^m(U) = \OP \tilde{\s}^m(U)$.  We extend $g$ to an
isomorphism
\be
\tilde{g}: \widetilde{\Psi}^0/\Psi^{-\infty}(U) \to \widetilde{\Psi}^0 /
\Psi^{-\infty}(V)
\ee
by defining for $C \in \cx$ and $P \in \Psi^0(U)$
\be
\tilde{g}(C + P) := C + g(P).
\ee
Observe $\tilde{g}$ induces an algebra isomorphism  
\be
g_0:\tilde{\s}^0/\tilde{\s}^{-1}(U) \to \tilde{\s}^0/\tilde{\s}^{-1}(V) .
\ee
Since $g_0$ takes maximal ideals to maximal ideals, we can define a
map
\be
\kappa: U  \to V .
\ee
First note that since $g_0: \Ci_c(U) \to \Ci_c(V)$,
\be
g_0(\mathcal{M}_{\partial U}) = \mathcal{M}_{\partial
  V}.
\ee 
Then for general $\rho \in U$, define $\kappa: U \to V$ by 
\be
g_0 (\mathcal{M}_{\rho}) = \mathcal{M}_{\kappa(\rho)}.
\ee
By applying $g_0^{-1}$, we immediately see $\kappa$ is bijective. 

Now for $p\in \tilde{\s}^0(U) $ and $\rho \in U$, observe
\be
p - p(\rho) \cdot 1 \in \mathcal{M}_{\rho}
\ee
implies
\be
g(p) - p(\rho) \cdot 1  \in \mathcal{M}_{\kappa(\rho)}.
\ee
Thus
\be
g(p) \left( \kappa(\rho)  \right) = p( \rho) 
\ee
for every $\rho \in U$ implies
\be
g(p) =  p \circ \kappa^{-1}.
\ee

For each $\rho \in U$, let $(x, \xi)$ be local coordinates for $X$ in a
neighbourhood of $\rho$ which does not meet $\partial U$.  Choosing a
suitable cutoff $\chi_\rho$ equal to $1$ near $\rho$, the $\chi_\rho x_j$
and $\chi_\rho \xi_k$ are {\it approximate coordinates} near $\rho$: 
\be
&& \chi_\rho x_j, \chi_\rho \xi_k \in \s^0(U) \text{ for all }
j,k; \\
&& \chi_\rho x_j = x_j, \,\, \chi_\rho \xi_k = \xi_k \text{ near }
\rho.
\ee
Thus
\be
(\chi_\rho x_j) \circ \kappa^{-1} \in \s^0(V) ,
\ee
and similarly for $\chi_\rho \xi_j$ for all $j$.  Composing with
inverse coordinate functions in a neighbourhood of $\kappa(\rho)$
implies $\kappa^{-1}$ is smooth on $U$.  The same argument applied to $g^{-1}$
shows $\kappa$ is smooth on $V$, hence a
diffeomorphism. 
\end{proof}

\begin{lemma}
The diffeomorphism $\kappa$ constructed in Lemma
\ref{kappa-construction} is symplectic.
\end{lemma}
\begin{proof}

Observe $\Psi^0/ \Psi^{-\infty}(U)$ is a Lie algebra with brackets $ih^{-1} [\cdot,
  \cdot]$, and $g$ induces a Lie algebra isomorphism with $\Psi^0/\Psi^{-\infty}(V)$.
$\s^0(U)$ is a Lie algebra with brackets $ \{
\cdot, \cdot \}$, hence $g_0$ is a Lie algebra isomorphism $\s^0/\s^{-1}(U)
\to \s^0/ \s^{-1}(V)$.  Let $a, b \in \s^0(U)$ and calculate
\be
 g_0( \{ a,  b\}) = 
 \left\{ g_0( a), g(  b) \right\} ,
\ee
or
\be
 (\{a,b\}) \circ \kappa^{-1} =  \{ a \circ \kappa^{-1}, b \circ \kappa^{-1} \}.
\ee
Letting $a$ and $b$ run through local approximate coordinates implies $\kappa^{-1}$ is symplectic.
\end{proof}

Now fix $\widetilde{U} \Subset U$, and let 
\be
F: L^2(X) \to L^2(Y) 
\ee
be the $h$-Fourier integral operator associated to
$\kappa|_{\tU}$ as in Proposition \ref{global-egorov}.  We define an automorphism of
$\Psi^0/\Psi^{-\infty}(\widetilde{U})$, $g_1$, by
\ben
\label{g-1-def}
g_1 (P) = F^{-1}g(P)F.
\een
Observe $g_1$ is both order-preserving and preserves principal
symbol.

\begin{lemma}
Suppose 
\be
g_1 : \Psi^0/\Psi^{-\infty}(\widetilde{U}) \to
\Psi^0/\Psi^{-\infty}(\widetilde{U})
\ee
is an order-preserving automorphism which preserves principal symbol.
Then there exists $B \in
\Psi^0(\widetilde U)$, elliptic on $\widetilde{U}$ such that
\ben
\label{g-1-def-2}
g_1(P) = BPB^{-1} \bmod \O(h^\infty)
\een
for every $P \in \Psi^0/\Psi^{-\infty}(U)$.
\end{lemma}

\begin{proof}
The proof will be by induction.  We drop the dependence on $\widetilde{U}$ since
the lemma is concerned with automorphisms.  Suppose for $l \geq 1$ we have for
every $m$ and every $P \in \Psi^m$
\be
g_1(P) - P \in \Psi^{m-l}.
\ee
This induces a map 
\be
\beta: \s^m/ \s^{m-1}  \to \s^{m-l}/\s^{m-l-1} ,
\ee
which, using the Weyl composition formula \eqref{Weyl-comp}, satisfies
\be
&& \text{ (i) } \beta(pq) = \beta(p) q + p \beta(q); \\
&& \text{(ii) } \beta(\{p,q\}) = \{ \beta(p), q\} + \{ p, \beta(q)\}.
\ee
Consider the action of $\beta$ on $\s^0$, and observe from
property (i) above, for $p,q \in \s^0$,
\be
\beta(pq) = \beta(p) q + p \beta(q) \in \s^{-l},
\ee
so $\beta$ is $h^{l}$ times a derivation on $\s^0$.


For any $\rho \in U$, we choose coordinates $(x, \xi)$ near $\rho$,
and a cutoff $\chi_\rho$ which is equal to $1$ near $\rho$ and
compactly supported in $U$.  Then $\chi_\rho x_j$ and $\chi_\rho
\xi_j$ become approximate coordinates which are equal to $x_j$ and
$\xi_j$ near $\rho$ but are in $\s^0$.  Near $\rho$, $\beta$ takes the
form
\be
\beta =h^l \sum_j \left( \gamma_j(x, \xi) \partial_{x_j} + \delta_j(x,
\xi) \partial_{\xi_j} \right),
\ee
where $\gamma_j = \beta( \chi_\rho x_j)$ and $\delta_j =
\beta(\chi_\rho \xi_j)$.  Using property (ii) above, we have near $\rho$
\be
\beta(\{\chi_\rho x_j, \chi_\rho \xi_k \}) = \beta(\{\chi_\rho x_j
, \chi_\rho x_k\}) = \beta( \{ \chi_\rho \xi_j, \chi_\rho \xi_k \} ) =
0
\ee
which implies
\be
\frac{\partial \gamma_j}{\partial x_k} = - \frac{\partial
  \delta_k}{\partial \xi_j}, \,\, \frac{\partial \gamma_j}{\partial
  x_k} = \frac{\partial \gamma_k}{\partial x_j}, \text{ and }
\frac{\partial \delta_j}{\partial \xi_k} = \frac{\partial
  \delta_k}{\partial \xi_j}.
\ee
Thus there exists a locally defined smooth function $f$ such that 
\be
\gamma_j = \frac{\partial f}{\partial \xi_j} \text{ and } \delta_k = -
\frac{\partial f}{\partial x_k},
\ee
and locally
\be
h^{-l}\beta = H_f.
\ee
Define a smooth function $b$ by
\be
b = \exp(-i\chi f),
\ee
for a cutoff $\chi$ which is identically $1$ on $\widetilde{U}$ with support in
$U$, so that $df = i db/b$ on $\widetilde{U}$, and locally
\be
\beta = h^l H_{i \log b}.
\ee
Let $B = \Op_h^w(b)$ and observe the principal symbol of 
\be
B^{-1}P B - P = B^{-1}[P,B]
\ee
in the $\s^0 (\widetilde U)$ calculus is
\be
\frac{h}{i} b^{-1} \{p,b\} = h H_{i\log b} (p).
\ee
For the base case of our induction, if $P \in \s^m(
\widetilde U)$, then 
\be
g_1(P) - B^{-1}PB \in \s^{m-2}( \widetilde U),
\ee
so that
\be
B  g_1 (P) B^{-1} - P \in \s^{m-2}( \widetilde U).
\ee
Replace $g_1(P)$ with $B g_1 B^{-1}$.  

Now for the purposes of induction, assume 
\be
g_1 (P) - P \in \s^{m-l}( \widetilde U),
\ee
and apply the above argument to get $B_l \in \s^{-l}$ so that
\be
g_1 (P) - B_l^{-1} P B_l \in \s^{m-l-1}( \widetilde U).
\ee
Then replacing $ g_1 (P)$ with $B_l g_1 (P) B_l^{-1}$
finishes the induction.  Thus there exists $B \in \s^0(\widetilde U)$ so
that
\be
B g_1 (P) B^{-1} = P \bmod \O(h^\infty).
\ee
\end{proof}

Theorem \ref{main-theorem} now follows immediately from applying
Proposition \ref{AF=FB} to \eqref{g-1-def} and \eqref{g-1-def-2}.

\qed


\begin{thebibliography}{\hspace{1cm}}


\bibitem[Ale]{Ale}{\sc Alexandrova, I.}  Semi-Classical Wavefront Set
  and Fourier Integral Operators. To appear in {\it Can. J. Math.}
{\tt http://personal.ecu.edu/alexandrovai/hfioa.pdf}



\bibitem[Ch1]{Ch}{\sc Christianson, H.}  Semiclassical
  Non-concentration Estimates near a Closed Loxodromic Orbit.  {\it J. Funct. Anal.}  {\bf 262} (2007), no. 2, 145--195. 


\bibitem[Ch2]{Ch2}{\sc Christianson, H.}  Quantum Monodromy and
  Non-concentration near a Closed Semi-hyperbolic Orbit.  {\it
    preprint.} \\
{\tt http://www-math.mit.edu/$\sim$hans/papers/qmnc.pdf }

\bibitem[DiSj]{DiSj}{\sc Dimassi, M. and Sj\"{o}strand, J.}  {\it Spectral Asymptotics in the Semi-classical Limit}.  
Cambridge University Press, Cambridge, 1999.

\bibitem[Dui]{duistermaat}{\sc Duistermaat, J. J.}  {\it Fourier Integral Operators}.  Birkh\"{a}user, Boston, 1996.

\bibitem[DS]{DS}{\sc Duistermaat, J. J., and Singer, I. M.}
  Order-Preserving Isomorphisms Between Algebras of
  Pseudo-Differential Operators.  {\it Commun. Pure Appl. Math.}  {\bf
    29}, 1976, 39-47.

\bibitem[EvZw]{EvZw}{\sc Evans, L.C. and Zworski, M.}  {\it Lectures on Semiclassical Analysis}.  \\
{\tt http://math.berkeley.edu/$\sim$evans/semiclassical.pdf}.


\bibitem[Hor1]{Hor1}{\sc H\"ormander, L.}  Fourier integral
  operators. I. {\it Acta Math.} {\bf 127} (1971), no. 1-2, 79--183. 

\bibitem[Hor2]{Hor}{\sc H\"ormander, L.}  {\it The analysis of linear
    partial differential operators. IV. Fourier integral operators.}
  Grundlehren der Mathematischen Wissenschaften [Fundamental
  Principles of Mathematical Sciences], {\bf 275}. Springer-Verlag, Berlin, 1985. 

\bibitem[Mel]{Mel}{\sc Melrose, R.}  Transformation of boundary problems. Acta Math. 147 (1981), no. 3-4, 149--236. 

\bibitem[Zel]{Zel}{\sc Zelditch, S.}  Quantum Maps and Automorphisms.
        {\it The breadth of symplectic and Poisson geometry,
          Prog. Math.,232}.  2005, 623-654.  
\end{thebibliography}
\end{document}